\documentclass[11pt,leqno]{article}
\usepackage{amsthm,amsfonts,amssymb,amsmath,oldgerm}
\usepackage{epsfig}
\numberwithin{equation}{section}
\usepackage[thinlines]{easybmat}

\renewcommand\a{\alpha}
\renewcommand\b{\beta}

\def\l{\lambda}

\renewcommand\a{\alpha}

\renewcommand\b{\beta}

\newcommand\C{\mathbb C}

\def\l{\lambda}

\newcommand\br{\begin{remark}}
\newcommand\er{\end{remark}}
\newcommand\bp{\begin{pmatrix}}
\newcommand\ep{\end{pmatrix}}
\newcommand\be{\begin{equation}}
\newcommand\ee{\end{equation}}
\newcommand\ba{\begin{equation}\begin{aligned}}
\newcommand\ea{\end{aligned}\end{equation}}
\newcommand\ds{\displaystyle}

\newcommand{\bap}{\begin{app}}
\newcommand{\eap}{\end{app}}
\newcommand{\begs}{\begin{exams}}
\newcommand{\eegs}{\end{exams}}
\newcommand{\beg}{\begin{example}}
\newcommand{\eeg}{\end{exaplem}}
\newcommand{\bpr}{\begin{proposition}}
\newcommand{\epr}{\end{proposition}}
\newcommand{\bt}{\begin{theorem}}
\newcommand{\et}{\end{theorem}}
\newcommand{\bc}{\begin{corollary}}
\newcommand{\ec}{\end{corollary}}
\newcommand{\bl}{\begin{lemma}}
\newcommand{\el}{\end{lemma}}
\newcommand{\bd}{\begin{definition}}
\newcommand{\ed}{\end{definition}}
\newcommand{\brs}{\begin{remarks}}
\newcommand{\ers}{\end{remarks}}

\newtheorem{theo}{Theorem}[section]

\newtheorem{exams}[theo]{Examples}

\numberwithin{equation}{section}

\newcommand{\RR}{{\mathbb R}}

\newcommand{\ZZ}{{\mathbb Z}}

\newcommand{\CC}{{\mathbb C}}

\newtheorem{theorem}{Theorem}[section]
\newtheorem{proposition}[theorem]{Proposition}
\newtheorem{corollary}[theorem]{Corollary}
\newtheorem{lemma}[theorem]{Lemma}
\newtheorem{definition}[theorem]{Definition}

\newtheorem{example}[theorem]{Example}
\newtheorem{remark}[theorem]{Remark}

\newcommand{\beq}{\begin{equation}}
\newcommand{\eeq}{\end{equation}}
\newcommand{\bs}{\begin{split}}
\newcommand{\es}{\end{split}}

\title{Pointwise nonlinear stability of nonlocalized modulated periodic reaction-diffusion waves}

\author{\sc \small
Soyeun Jung \thanks{Kongju National University, Gongju-si, Chungcheongnamdo 314-701, Korea; soyjung@kongju.ac.kr: This work was supported by the research grant of the Kongju National University in 2015.}
~~~~~
Kevin Zumbrun \thanks{Indiana University, Bloomington, IN 47405, USA;
kzumbrun@indiana.edu
 }}

\begin{document}

\maketitle

\begin{abstract}
In this paper, extending previous results of \cite{J1}, we obtain pointwise nonlinear stability of periodic traveling reaction-diffusion waves, assuming spectral linearized stability, under nonlocalized perturbations. More precisely, we establish pointwise estimate of nonlocalized modulational perturbation under a small initial perturbation consisting of a nonlocalized modulation plus a localized perturbation decaying algebraically.

\end{abstract}

%

\section{Introduction } \label{introduction}
We consider a system of reaction-diffusion equations
\be\label{RD}
u_t=u_{xx}+f(u),
\ee
where $(x,t)\in \RR \times \RR^+,$  $u\in \RR^n$ and $f:\RR^n\rightarrow \RR^n$ sufficiently smooth. We assume that $u(x,t)=\bar{u}(x-ct)$ is a traveling wave solution of the system \eqref{RD} with a constant speed $c$ and the profile $\bar u(\cdot)$ satisfies $\bar u(\cdot)=\bar u(\cdot+1)$. In other words, $\bar u(x)$ is a stationary $1-$periodic solution of the PDE
\be\label{reaction-diffusion eq. 1}
u_t=u_{xx}+cu_x+f(u).
\ee
In \cite{J1}, the first author established pointwise Green function bounds on the linearized operator about the underlying solution $\bar u$ and obtained pointwise nonlinear stability of $\bar u$ by estimating the localized modultional perturbation $v(x,t)=\tilde u(x-\psi(x,t),t)-\bar u(x)$ $(h_0=\psi(x,0)=0)$ under small initial perturbations $v(x,0)=\tilde u(x,0)-\bar u(x)$ decaying algebraically for nearby solutions $\tilde u$ of \eqref{reaction-diffusion eq. 1}.

In the present paper, we study the pointwise nonlinear stability of $\bar u$ of \eqref{reaction-diffusion eq. 1} under small perturbations consisting of a nonlocalized modulation ($h_0(x)=\psi(x,0)$ does not decay algebraically, but $\partial_x h_0$ decays algebraically) plus a localized perturbation $v(x,0)=\tilde u(x-h_0(x),0)-\bar u(x)$ ($v(x,0)$ decays algebraically). Johnson, Noble, Rodrigues and Zumbrun showed $L^p$-nonlinear stability under such nonlocalized modulational perturbations ($h_0 \not \in L^1$, but $\partial_x h_0 \in L^1$) for systems of reaction-diffusion equations in \cite{JNRZ1} and of conservation laws in \cite{JNRZ3}. Sandsteds, Scheel, Schneider, and Uecker obtained similar results by rather different methods for systems of reaction-diffusion equations in \cite{SSSU}.

Similarly as in \cite{JNRZ1,JNRZ3}, here,  we determine an appropriate nonlocalized modulation $\psi(x,t)$ by an adaption of
the basic nonlinear iteration scheme developed in \cite{JZ}.
However, in the absence of cancellation estimates afforded by Hausdorff-Young and Parseval inequalities, we find it necessary to decompose the solution a bit differently than was done in \cite{JNRZ1} in order to estimate sharply
the key ``modulation'' part of the linearized solution operator in reponse to modulational-type data (see Remark \ref{difference}),
and to estimate this modulational part essentially ``by hand.''
This is the main new difficulty in our analysis beyond those carried out in \cite{J1,JNRZ1}.

\subsection{Preliminaries}

We first linearize the PDE \eqref{reaction-diffusion eq. 1} about a stationary $1-$periodic solution $\bar u$ so that we obtain the eigenvalue problem
\be\label{sp1}
\l v=Lv:=(\partial_x^2+c\partial_x+df(\bar{u}))v,
\ee
operating on $L^2(\RR)$ with densely defined domains $H^2(\RR)$. Here, $v$ is considered as a perturbation of $\bar u$ defined by $v(x,t)=\tilde u(x,t)-\bar u(x)$ for nearby solutions $\tilde u$. To characterize the $L^2(\RR)$-spectrum of $L$ (denoted by $\sigma_{L^2(\RR)}(L)$), we rewrite \eqref{sp1} as the following linear ODE system
\be\label{linear ODE system}
V_x=\mathbb{A}V,  \quad \text{where $V=\bp v \\ v_x \ep$ and  $\mathbb{A}=\mathbb{A}(x,\l)=\bp 0 & I \\ \l I-df(\bar u) & -cI \ep$}.
\ee
Since all coefficients of $L$ are $1$-periodic, $\mathbb{A}(x+1,\l)=\mathbb{A}(x, \l)$; so by Floquet theory, the fundamental matrix solution $\Phi(x,\l)$ of the linear system \eqref{linear ODE system} is
\be
\Phi(x,\l)=P(x,\l)e^{R(\l)x},
\notag
\ee
where $R(\l) \in \C^{n\times n}$ is a constant matrix and $P(x,\l) \in \CC^{n\times n}$ is a periodic matrix, $P(x,\l)=P(x+1,\l)$. In fact, for each eigenvalue $\mu$ (referred to as the Floquet exponent) of $R(\l)$, there is a solution to \eqref{linear ODE system} of the form $V(x,\l)=e^{\mu x}W(x, \l)$, where $W$ is $1$-periodic in $x$.  Thus, any non-trivial solution $V$ to the system \eqref{linear ODE system} does not lie in $L^2(\RR)$, which means that the $L^2(\RR)$-spectrum of the linear operator $L$ must be entirely essential. Moreover, $\l \in \sigma_{L^2(\RR)}(L) $ if and only if $R(\l)$ is not hyperbolic; thus there is a solution to \eqref{sp1} of the form $v=e^{i\xi x}w(x)$ for some neutral eigenvalue $i\xi \in \sigma(R(\l))$ and $1-$periodic function $w$. Here, $\xi \in [-\pi, \pi)$ is uniquely defined mod $2\pi$. Plugging $v(x)=e^{i\xi x}w(\xi, x)$ into \eqref{sp1}, we define the Bloch operators, for each $\xi \in [-\pi, \pi)$,
\be\label{BO}
L_{\xi}:=(\partial_x+i\xi)^2 + c(\partial_x+i\xi)+df(\bar u)
\ee
operating on $L^2_{per}([0,1])$ with densely defined domain $H^2_{per}([0,1])$. Indeed, $L_{\xi}$ satisfies $L_{\xi}w=\l(\xi)w$ and $w$ is 1-periodic in $x$. Moreover, for each $\xi\in [-\pi, \pi)$, $L_{\xi}$ has a compact resolvent and $\sigma_{L^2(\RR)}(L)= \ds \bigcup_{\xi \in [-\pi,\pi)} \sigma_{L^2_{per}([0,1])}(L_{\xi})$. In other words, for each $\xi \in [-\pi, \pi)$, $L_{\xi}$ operating on $L^2_{per}([0,1])$ has only  point spectrum, whereas $L$ operating on $L^2(\RR)$ has only essential spectrum.

\smallskip

We now define the standard diffusive spectral stability conditions (following \cite{S1, S2}).  We first notice that $0$ is an eigenvalue of $L_0$ because $L\bar u'=0$ and $\bar u'$ is 1-periodic. Throughout our analysis we assume the following conditions:

\smallskip

(D1) $\sigma_{L^2(\RR)}(L) \subset \{ \l \in \CC : R(\l) <0\}\cup \{0\}$.

\smallskip

(D2) $\l=0$ is a simple eigenvalue of $L_0$.

\smallskip

(D3) There exists a constant  $\theta>0$ such that $R\sigma(L_\xi) \leq -\theta |\xi|^2$ for all $\xi \in [-\pi, \pi)$.

\smallskip

\noindent As we mentioned above, $0$ is not an isolated eigenvalue of $L$ but a member of a continuous curve of essential spectrum, so there is no spectral gap between $0$ and the rest of the spectrum. This is the reason why we could not use the stability methods used for other types of traveling reaction diffusion waves such as front or pulse. This difficulty was overcome in Swift-Hohenberg equation in \cite{S1, S2} by using the above diffusive spectral stability. Moreover, from the above three conditions, the eigenvalue of $L_{\xi}$ is analytic at $\xi=0$; so the eigenvalue of $L_{\xi}$  bifurcating from $0$ at $\xi=0$ has the following expression (see \cite{J1, JNRZ1})
\be\label{eigenvalue from 0}
\l(\xi)=-ia\xi - b\xi^2 + O(|\xi|^3) \quad \text{for sufficiently small $|\xi|$,}
\ee
where $a \in \RR$ and $b>0$.

\subsection{Bloch transform}

We now recall the Bloch transform, as described for example in \cite{J1, JNRZ1, JNRZ2, JNRZ3}. By the inverse Fourier transform, we have for any $g \in L^2(\RR)$,
\be
g(x)= \frac{1}{2\pi}\int_{-\pi}^{\pi } e^{i\xi x}\check g(\xi,x) d\xi,
\notag
\ee
where $\check g(\xi,x) = \ds\sum_{j \in \ZZ}e^{i2\pi jx}\hat g(\xi+2\pi j)$ (referred to as the Bloch transform) and $\hat g(\cdot)$ denotes the Fourier transform of $g$ with respect to $x$. By the definition of $L_{\xi}$ in \eqref{BO}, $L(e^{i\xi x}f)=e^{i\xi x}(L_{\xi}f)$ for periodic functions $f$ and $\check g(\xi,x)$ is 1-periodic in $x$; so we have the Bloch solution formula for the linear operator $L$ in \eqref{sp1}
\be\label{inverse BF}
S(t)g(x):=e^{Lt}g(x)= \frac{1}{2\pi}\int_{-\pi }^{\pi } e^{i\xi x}e^{L_\xi t}\check g(\xi,x)
d\xi,
\ee
for any $g \in L^2(\RR)$.

\subsection{Main result} \label{main result}

With these preparations, we state the main theorem of this paper. Here, and throughout the paper, $h_{\infty}$ and $h_{-\infty}$ denote the end states of the initial modulation $h_0(x)$ as $x \rightarrow \infty$ and $x \rightarrow -\infty$, respectively, and $h_{\pm \infty}$ denotes a piecewise constant function defined by
\be\label{h infty}
h_{\pm \infty}(x):=\left\{
  \begin{array}{ll}
    h_{-\infty}, & \hbox{$x \leq 0$;} \\
    h_{+\infty}, & \hbox{$x>0$.}
  \end{array}
\right.
\ee

\begin{theorem}\label{main theorem}
 Suppose that a stationary 1-periodic solution $\bar u(x)$ of \eqref{reaction-diffusion eq. 1} satisfies spectral stability conditions $(D1)\sim(D3)$. For $r\geq\frac{3}{2}$ and sufficiently small $E_0>0$, we assume that the initial data $\tilde u_0(x)$ and $h_0(x)$ satisfy
 \be\label{initial data}
\begin{split}
 |\tilde u_0(x -h_0(x))-\bar u(x)| + & \sum_{k=1, 2} |\partial_x^k h_0(x)|+|h_0(x)-h_{\pm \infty}|  \leq E_0(1+|x|)^{-r}, \\
|h_{+\infty}|=|h_{-\infty}|\leq E_0  & \quad \text{and} \quad v_0:=\tilde u_0(x -h_0(x))- \bar u(x)  \in H^2(\RR).
\end{split}
\ee
Then for all initial data $\tilde u_0$ satisfying \eqref{initial data}, the corresponding solution $\tilde u(x,t)$ to \eqref{reaction-diffusion eq. 1} satisfies
\be\label{MR}
|\tilde u(x -\psi(x,t),t)-\bar u(x)|
 \leq CE_0\Big[(1+|x-at|+\sqrt t)^{-r} + (1+t)^{-\frac{1}{2}}e^{-\frac{|x-at|^2}{M't}}\Big]\\
\ee
and
\be
|\partial_x \psi(x,t)|, |\partial_t \psi(x,t)|
 \leq CE_0\Big[(1+|x-at|+\sqrt t)^{-r} + (1+t)^{-\frac{1}{2}}e^{-\frac{|x-at|^2}{M't}}\Big],
\ee
for an appropriate modulation $\psi(x,t) \in W^{2,\infty}$ with $\psi(x,0)=h_0(x)$ (which is determined in Section \ref{section NIS}). Here, $M'>M>0$ is a sufficiently large number  ($M$ denotes the constant in Theorem \ref{PGB}) and the constant $a$ is from \eqref{eigenvalue from 0}.
\end{theorem}

\begin{remark}
This extends the result of \cite{J1} to nonlocalized modulations. If $h_0(x)$ decays algebraically, we obtain the same result as \cite{J1}.  Here, compared with the localized case \cite{J1}, one can see that nonlocalized modulated perturbations decay at a slower, heat kernel rate. Moreover, the initial data conditions \eqref{initial data} satisfy the conditions in \cite{JNRZ1}, so that integrating the bound in \eqref{MR} with respect to $x$ gives the same $L^p$-bound as in \cite{JNRZ1} which is $(1+t)^{-\frac{1}{2}(1-\frac{1}{p})}$ for all $2\leq p \leq \infty$.
\end{remark}

\begin{remark}\label{end states assumption}
Here, without loss of generality, one can assume $h_{-\infty}=-h_{+\infty}$. If not, $h_{-\infty}+c=-(h_{+\infty}+c) $ with $c=-\ds\frac{h_{+\infty}+h_{-\infty}}{2}$. This assumption is for handling nonlocalized functions in the Bloch transform framework.
\end{remark}

\begin{remark}
For the algebraically decaying data, $(1+|x|)^{-r}$ with $r>1$ is enough for linear stability. However, we need $r\geq\frac{3}{2}$ for nonlinear stability when we use the nonlinear iteration scheme in Section \ref{section NIS}.
\end{remark}

\subsection{Nonlinear perturbation equations and outline of the analysis} \label{Nonlinear Perturbation Equations}


We first look at the nonlinear equation of modulated perturbations of $\bar u$ to obtain the strategy of this paper. As mentioned in the previous section, we define the modulated perturbations
\be\label{definition of v}
v(x,t)=\tilde u(x-\psi(x,t),t)-\bar u(x)
\ee
for nearby solutions $\tilde u(x,t)$ to \eqref{reaction-diffusion eq. 1} and unknown function $\psi(x,t):\RR^2 \longrightarrow \RR$ with $\psi(x,0)=h_0(x)$ to be determined in Section \ref{section NIS}. This is exactly what we want to estimate in the main theorem; so, we first state the nonlinear perturbation equation about $v$ which is already established in \cite{JZ, JNRZ1}.

\begin{lemma}[Nonlinear perturbation equations, \cite{JZ, JNRZ1}] For $v$ defined in \eqref{definition of v} and the linear operator $L$ in \eqref{sp1}, we have
\be \label{nonlinear perturbation equation}
(\partial_t-L)v=-(\partial_t-L)\bar u'(x)\psi + Q+R_x+(\partial_x^2+\partial_t)Z+T,
\ee
where
\be \label{Q}
Q:=f(v(x,t)+\bar{u}(x))-f(\bar{u}(x))-df(\bar{u}(x))v=\mathcal{O}(|v|^2),
\ee
\be \label{R}
R:= -v\psi_t - v\psi_{xx}+  (\bar u_x +v_x)\frac{\psi_x^2}{1-\psi_x},
\ee
\be \label{S}
Z:= v\psi_x =O(|v| |\psi_x|),
\ee
and
\be \label{T}
T:=-\left(f(v+\bar{u})-f(\bar{u})\right)\psi_x=O(|v||\psi_x|).
\ee
\end{lemma}

\smallskip

We now briefly give the plan of this paper. Setting $v_0(x)=\tilde u_0(x -h_0(x))-\bar u(x)$ and $\mathcal{N}(x,t)=(Q+R_x+(\partial_x^2+\partial_t)Z+T)(x,t)$ in \eqref{nonlinear perturbation equation},  by Duhamel's principle, we have
\be\label{v}
v(x,t)=-\bar u'(x) \psi(x,t)+ e^{Lt}(v_0+\bar u'h_0) + \int_0^t e^{L(t-s)}\mathcal{N}(s)ds.
\ee
For localized data $v_0$, in order to estimate $e^{Lt}v_0$, we use the pointwise Green function bounds obtained in \cite{J1}; so we first recall one of the main theorems of \cite{J1} in Section \ref{Linear estimates for the localized data}. Since we defined $\psi(x,t)$ with $h_0(x)=\psi(x,0)=0$ in \cite{J1}, the main new ingredient in this paper compared to \cite{J1} is the pointwise linear behavior under modulational data $\bar u'h_0$. In other works, the main difficulty here is to estimate $e^{Lt}(\bar u'h_0)$ in terms of the localized data $|\partial_x h_0|$, $|\partial_x^2 h_0|$ or $|h_0-h_{\pm \infty}|$ in Section \ref{Linear estimates for the nonlocalized data}. After we estimate the linear level, we define an appropriate $\psi(x,t)$ with $\psi(x,0)=h_0(x)$; so we finally obtain pointwise bounds of $v$ by the nonlinear iteration scheme in Section \ref{section NIS} and Section \ref{section NS}.


\subsection{Discussion and open problems}

Compared with \cite{JNRZ1} ($L^p$-stability estimates for nonlocalized modulations),  the assumptions $|\tilde u_0(x -h_0(x))-\bar u(x)|$, $\sum_{k=1, 2} |\partial_x^k h_0(x)| \leq E_0(1+|x|)^{-r} $ in \eqref{initial data} are very natural for pointwise estimates.
However, the assumptions on $h-h_{\pm\infty}$ might appear unfamiliar. The reason for these
is that it is still an open problem how to establish pointwise estimates on the linearized solution operator directly from the Bloch representation
\be
e^{Lt}g(x)= \frac{1}{2\pi}\int_{-\pi }^{\pi } e^{i\xi x}e^{L_\xi t}\check g(\xi,x)
d\xi,
\notag
\ee
for $|x|>>Ct$ (sufficiently large $C>0$) even for the localized data $g$.
These additional assumptions allow us to obtain estimates by a different route.
This is also the reason why we obtained pointwise Green function bounds $G(x,t;y)$ on the linear operator $L$ in \cite{J1, J2} (the cases of localized modulations), without use of the Bloch representation for $|x-y|>>Ct$.
Thus, if we could find out how to handle the Bloch solution operator for $|x|>>Ct$,
this would be a nice improvement both in Theorem \ref{main theorem}, and in the analysis of the previous work \cite{J1, J2}.

Another, very interesting,
open problem is to determine not only pointwise derivative decay of the modulation $\psi$, but also its pointwise behavior to lowest order,
similarly as done in the $L^p$ context in \cite{JNRZ2}.
Indeed, this might be a route also to the elimination of hypotheses on decay of $h-h_{\pm \infty}$, since subtracting off this principal behavior would
leave only localized terms more amenable to pointwise estimates.  However, as noted earlier, our definition of the phase, being adapted to the pointwise analysis,
is somewhat different from that in \cite{JNRZ1,JNRZ2}, and so we cannot immediately apply the earlier analysis to obtain such a result.


\section{Linear estimates for the localized data $v_0$}\label{Linear estimates for the localized data}

In this section, we recall the pointwise Green function bounds of the linear operator $L$ from \cite{J1}.
\begin{theorem}[Pointwise Green function bounds, \cite{J1}]\label{PGB}
The Green function $G(x,t;y)$ for the evolution equations $(\partial_t-L)v=0$ for linear operator \eqref{sp1} satisfies the estimates:
\be
G(x,t;y)=\bar u'(x)E(x,t;y)+\tilde G(x,t;y),
\notag
\ee
where
\be
E(x,t;y)=\frac{1}{\sqrt{4\pi bt}}e^{-\frac{|x-y-at|^2}{4bt}}\tilde q(y,0)\chi(t),
\notag
\ee

\be
|\tilde G(x,t;y)| \lesssim \Big((1+t)^{-1}+t^{-\frac{1}{2}}e^{-\eta t}\Big) e^{-\frac{|x-y-at|^2}{Mt}}
\notag
\ee
and
\be
|\tilde G_y(x,t;y)| \lesssim t^{-1} e^{-\frac{|x-y-at|^2}{Mt}},
\notag
\ee
uniformly on $t \geq 0$, for some sufficiently large constant $M>0$ and $\eta>0$. Here $\tilde q$ is the periodic left eigenfunction of $L_0$ at $\l=0$ and $\chi(t)$ is a smooth cutoff function such that  $\chi(t)=0$ for $0 \leq t \leq \frac{1}{2}$ and $\chi(t)=1$ for $t \geq 1$.
\end{theorem}

\begin{remark}
In the main theorem in \cite{J1}, the Green function $G(x,t;y)$ has no cutoff function $\chi(t)$. However, there is no difference between Theorem \ref{PGB} and the original theorem in \cite{J1} because $(1-\chi(t))\ds\bar u'(x)\frac{1}{\sqrt{4\pi bt}}e^{-\frac{|x-y-at|^2}{4bt}} \tilde q(y,0)$ is included in $\tilde G(x,t;y)$.
\end{remark}

In \eqref{v}, for the localized data $v_0$, we estimate $e^{Lt}v_0$ as
\be
\begin{split}
(e^{Lt}v_0)(x)
&=\bar u'(x)\int_{-\infty}^{\infty} E(x,t;y)v_0(y)dy +  \int_{-\infty}^{\infty} \tilde G(x,t;y)v_0(y)dy. \\
\end{split}
\notag
\ee
Here, we assume algebraic decay of the initial localized data $v_0$; so we need to look at the linear behavior of $L$ under the algebraically decaying data which was completed in \cite{J1, HZ}.  We re-prove it here in the following lemma because it is used throughout this paper.

\begin{lemma}[\cite{J1}, \cite{HZ}] \label{algebraically decaying data} Let $r>1$. Then for any $x \in \RR$ and $t \geq 0$,
\be\label{agebraic decay}
\int_{-\infty}^\infty t^{-\frac{1}{2}}e^{-\frac{|x-y-at|^2}{bt}}(1+|y|)^{-r} dy
\lesssim (1+|x-at|+\sqrt t)^{-r}+(1+t)^{-\frac{1}{2}} e^{-\frac{|x-at|^2}{4bt}}.
\ee
\end{lemma}
\begin{proof}
If $x-at=0$, then it is trivial because
\be
\int_{-\infty}^\infty t^{-\frac{1}{2}}e^{-\frac{|x-y-at|^2}{bt}}(1+|y|)^{-r} dy \lesssim (1+t)^{-\frac{1}{2}} =  (1+t)^{-\frac{1}{2}} e^{-\frac{|x-at|^2}{bt}}.
\notag
\ee
We assume $x-at \neq 0$. Since $|x-y-at| \geq ||x-at|-|y||$,
\be
\begin{split}
\int_{-\infty}^\infty t^{-\frac{1}{2}}e^{-\frac{|x-y-at|^2}{bt}}(1+|y|)^{-r} dy
& \leq \int_{-\infty}^\infty t^{-\frac{1}{2}}e^{-\frac{||x-at|-|y||^2}{bt}}(1+|y|)^{-r} dy \\
& = 2\int_{0}^\infty t^{-\frac{1}{2}}e^{-\frac{||x-at|-|y||^2}{bt}}(1+y)^{-r} dy \\
& \approx \int_{0}^{\frac{|x-at|}{2}} (\cdots)+ \int_{\frac{|x-at|}{2}}^{2|x-at|} (\cdots)+ \int_{2|x-at|}^\infty (\cdots).
\notag
\end{split}
\ee
Noting first that the left-hand side of \eqref{agebraic decay} and $\ds\int_{0}^{\infty} (1+y)^{-r} dy$ with $r>1$  are bounded,
\be
\int_{0}^{\frac{|x-at|}{2}} (\cdots) + \int_{2|x-at|}^\infty (\cdots) \lesssim (1+t)^{-\frac{1}{2}} e^{-\frac{|x-at|^2}{4bt}} \int_{0}^{\infty} (1+|y|)^{-r} dy
\lesssim (1+t)^{-\frac{1}{2}} e^{-\frac{|x-at|^2}{4bt}}.
\notag
\ee
We now estimate $\ds \int_{\frac{|x-at|}{2}}^{2|x-at|} (\cdots)$ in two cases. If $|x-at| \leq \sqrt t$, then
$$e^{-\frac{|x-at|^2}{4bt}} = e^{-\frac{1}{4b}(\frac{|x-at|}{\sqrt t})^2} \geq e^{-\frac{1}{4b}} > 0; $$
so
\be
\int_{\frac{|x-at|}{2}}^{2|x-at|} (\cdots) \lesssim (1+t)^{-\frac{1}{2}} \lesssim (1+t)^{-\frac{1}{2}} e^{-\frac{|x-at|^2}{4bt}}.
\notag
\ee
If $|x-at| > \sqrt t$, then
\be
\int_{\frac{|x-at|}{2}}^{2|x-at|} (\cdots) \lesssim (1+|x-at|)^{-r} \lesssim (1+2|x-at|)^{-r} \lesssim  (1+|x-at|+\sqrt t)^{-r}.
\notag
\ee

\end{proof}

\section{Linear estimates for nonlocalized modulational data $\bar u'h_0$}\label{Linear estimates for the nonlocalized data}

For the modulational data $\bar u'h_0$, we recall \eqref{inverse BF} and decompose the solution operator $S(t)$ into
\be
S(t)=S_*^p(t) + \tilde S_*(t), \quad S_*^p(t)=\bar u's_*^p(t)
\notag
\ee
with
\be
s^p_*(t)(\bar u'h_0):=\int_{-\infty}^{\infty} e^{i\xi x}e^{(-ia\xi-b\xi^2)t} \hat{h_0}(\xi) d\xi
\notag
\ee
and
\be
\tilde S_*(t)(\bar u'h_0)=(S(t)- S_*^p(t))(\bar u'h_0).
\notag
\ee
We re-express $s_*^p(\bar u' h_0)$ as $\ds s_*^p(\bar u'h_0)
=IFT \Big(e^{(-ia\xi-b\xi^2)t} \Big)* h_0= \int_{-\infty}^\infty (4\pi bt)^{-\frac{1}{2}} e^{-\frac{|x-y-at|^2}{4bt}}h_0(y)dy$. Here, $IFT$ denotes the inverse Fourier transform and $*$ denotes the convolution. Similarly, we have
\be\label{first x derivative of psi}
\begin{split}
\Big|\partial_x s^p_*(t)(\bar u'h_0)\Big|
& =\Big|\int_{-\infty}^{\infty} e^{i\xi x}e^{(-ia\xi-b\xi^2)t} (i\xi \hat{h_0})(\xi) d\xi \Big|\\
& = \Big|\int_{-\infty}^{\infty} e^{i\xi x}e^{(-ia\xi-b\xi^2)t} \widehat{\partial_x h_0}(\xi) d\xi\Big| \\
& \leq \int_{-\infty}^\infty (4\pi bt)^{-\frac{1}{2}} e^{-\frac{|x-y-at|^2}{4bt}}|\partial_y h_0(y)|dy,
\end{split}
\ee

\be\label{second x derivative of psi}
\begin{split}
\Big| \partial_x^2 s^p_*(t)(\bar u'h_0)\Big|
& = \Big|\int_{-\infty}^{\infty} e^{i\xi x}e^{(-ia\xi-b\xi^2)t}(i\xi)^2 \hat{h_0}(\xi) d\xi \Big|\\
& \leq \int_{-\infty}^\infty (4\pi bt)^{-\frac{1}{2}} e^{-\frac{|x-y-at|^2}{4bt}}|\partial_y^2 h_0(y)|dy,
\end{split}
\ee
and
\be\label{first t derivative of psi}
\begin{split}
\Big|\partial_t s^p_*(t)(\bar u'h_0) \Big|
& = \Big|\int_{-\infty}^{\infty} e^{i\xi x}e^{(-ia\xi-b\xi^2)t} (-ia\xi-b\xi^2) \hat{h_0}(\xi) d\xi \Big| \\
& \lesssim \int_{-\infty}^\infty (4\pi bt)^{-\frac{1}{2}} e^{-\frac{|x-y-at|^2}{4bt}}\Big(|\partial_y h_0(y)|+|\partial_y^2 h_0(y)|\Big)dy.
\end{split}
\ee


\bigskip

We estimate $(\tilde S_*(t)(\bar u'h_0))(x) $ separately in two cases $|x| >> Ct$ and $|x|<<Ct$ for sufficiently large $C>0$.
Recalling \eqref{h infty}, we begin by estimating $(\tilde S_*(t)(\bar u'h_0))(x)$ for $|x|>>Ct$ by using the fact that $\bar u'h_{-\infty} $ and $\bar u'h_{+\infty}$ are stationary solutions of $S(t)$.

\bigskip

\begin{proposition}\label{prop S bigger than Ct}
Suppose $|(h_0-h_{\pm \infty})(x)| \leq E_0(1+|x|)^{-r}$  and $|h_{+\infty}|=|h_{-\infty}| \leq E_0$ for $r>1$ and sufficiently small $E_0>0$. Then if $|x|>>Ct$ for sufficiently large $C>0$,
\be
|(\tilde S_*(t)(\bar u' h_0))(x)| \lesssim E_0 \Big[ (1+|x-at|+\sqrt t)^{-r}+(1+t)^{-\frac{1}{2}}e^{-\frac{|x-at|^2}{M't}}   \Big ],
\notag
\ee
for a sufficiently large number $M'(>M>0)$ ($M$ denotes the constant in Theorem \ref{PGB}).
\end{proposition}

\begin{proof}
Let us first consider $x \leq 0 $. Since $\bar u'h_{-\infty}$ is a stationary solution of $S(t)$, $S(t)(\bar u'h_{-\infty})=\bar u'h_{-\infty}=S_*^p(t)(\bar u'h_{-\infty})$; and so  $\tilde S_*(t)(\bar u' h_0)
= S(t)(\bar u' (h_0-h_{-\infty}))-S_*^p(t)(\bar u' (h_0-h_{-\infty}))$.
From the pointwise bounds on the Green function of $S(t)$ in Theorem \ref{PGB},
\be
\begin{split}
& |S(t)(\bar u' (h_0-h_{-\infty}))| + |S_*^p(t)(\bar u' (h_0-h_{-\infty}))|  \\
& \lesssim  \int_{-\infty}^{0} t^{-\frac{1}{2}}e^{-\frac{|x-y-at|^2}{Mt}} |(h_0-h_{-\infty})(y)|dy  + \int_{0}^{\infty} t^{-\frac{1}{2}}e^{-\frac{|x-y-at|^2}{Mt}} |(h_0-h_{-\infty})(y)|dy  \\
& \lesssim  \int_{-\infty}^{\infty} t^{-\frac{1}{2}}e^{-\frac{|x-y-at|^2}{Mt}}|(h_0-h_{\pm \infty})(y)|dy +  E_0\int_{0}^{\infty} t^{-\frac{1}{2}}e^{-\frac{|x-y-at|^2}{Mt}}dy.\\
\end{split}
\notag
\ee
The first term is done from the assumption $|(h_0-h_{\pm \infty})(x)| \leq E_0(1+|x|)^{-r}$, $r>1$ and Lemma \ref{agebraic decay}. Since $x<<-Ct$ and $y\geq 0$, $|x-y|>>Ct$; and so
\be \label{inequalities}
\frac{|x-y-at|}{2t} \leq \frac{|x-y|}{t} \leq \frac{2|x-y-at|}{t},
\ee
for sufficiently large $C>0$. Thus,
\be
\begin{split}
\int_{0}^{\infty} t^{-\frac{1}{2}}e^{-\frac{|x-y-at|^2}{Mt}}dy \leq \int_{0}^{\infty} t^{-\frac{1}{2}}e^{-\frac{|x-y|^2}{4Mt}}dy \leq e^{-\eta t} e^{-\frac{|x|^2}{16Mt}} \int_{0}^{\infty} t^{-\frac{1}{2}}e^{-\frac{|x-y|^2}{16Mt}}dy \leq e^{-\eta t}e^{\frac{|x-at|^2}{M't}},
\end{split}
\notag
\ee
for some positive constant $M'>16M>0$.  Here, the last inequality is from the fact $|x|>>Ct$ and \eqref{inequalities} again. Similarly, we argue the case $x>0$ with a stationary solution $\bar u'h_{+\infty}$ instead of $\bar u'h_{-\infty}$.
\end{proof}

For the case $|x|<<Ct$, we first denote the right and left eigenfuctions of $L_{\xi}$ corresponding to $\l(\xi)$ by $q(\xi,x)$  and $\tilde q(\xi,x)$, respectively, for sufficiently small $|\xi|$. In particular, $q(0,x)=\bar u'(x)$ since $L_0\bar u'=0$. Moreover, for sufficiently small $|\xi|$, let $\Pi(\xi)(\cdot)=q(\xi)\langle\tilde q(\xi), \cdot \rangle_{L^2([0,1])}$ which is the eigenprojection onto the right-eigenspace, span$\{q(\xi)\}$ and  set $\tilde \Pi(\xi)=\text{I} - \Pi(\xi)$.  In order to estimate $\tilde S_*(t)$ for $|x|>>Ct$, we decompose again  $\tilde S_*(t)$ into $\tilde S_1(t)$ and $\tilde S_2(t)$ with
\be
\tilde S_1(t)(\bar u'h_0) :=\int_{-\pi}^{\pi} e^{i\xi x}e^{\l(\xi)t}\a(\xi)\langle \tilde\phi(\xi,\cdot), \bar{u}^\prime \check h_0(\xi,\cdot) \rangle_{L^2[0,1]} d\xi
- \int_{-\infty}^{\infty} e^{i\xi x}e^{(-ia\xi-b\xi^2)t} \hat{h_0}(\xi) d\xi
\notag
\ee
and
\be\label{tilde S2}
\begin{split}
\tilde S_2(t)(\bar u'h_0)
& := \int_{-\pi}^{\pi} e^{i\xi x}(1-\a(\xi)) e^{L_{\xi}t}( \bar u'\check{h_0}(\xi,x))d\xi \\
& \quad + \int_{-\pi}^{\pi} e^{i\xi x}\a(\xi)\tilde \Pi(\xi) e^{L_{\xi}t} ( \bar u'\check{h_0}(\xi,x))d\xi \\
& \quad +  \int_{-\pi}^{\pi} e^{i\xi x}\a(\xi)e^{\l(\xi)t}(q(\xi,x)-q(0,x)) \langle \tilde q(\xi,\cdot),  \bar u'\check{h_0}(\xi,\cdot)\rangle_{L^2[0,1]} d\xi, \\
\end{split}
\ee
where $\a(\xi)$ is a smooth cutoff function such that $\a(\xi)=1$ for sufficiently small $|\xi|$.

\begin{remark}\label{difference}
Comparing our decomposition of $S(t)=\bar u' s_*^p(t)+\tilde S^p(t)=\bar u' s_*^p(t) + \tilde S_1(t)+\tilde S_2(t)$ with the decomposition of $S(t)=\bar u' s^p(t) + S^p(t)$ in \cite{JNRZ1}, we see that $\bar u' s_*^p(t) + \tilde S_1(t) = \bar u' s^p(t) $ and $\tilde S_2(t)=S^p(t)$. Here, we set $s^p_*$ as the principal, Gaussian, part of the worst term at $j=0$ in
\be
\begin{split}
(s^p(t)(\bar u'h_0))(x)
=& \int_{-\pi}^{\pi} e^{i\xi x}e^{\l(\xi)t}\a(\xi)\langle \tilde\phi(\xi,\cdot), \bar{u}^\prime \check h_0(\xi,\cdot) \rangle_{L^2[0,1]} d\xi \\
= & \sum_{j \in \ZZ}\int_{-\pi}^{\pi}e^{i\xi x}e^{\l(\xi)t}\a(\xi)\langle \tilde\phi(\xi,\cdot)\bar{u}^\prime(\cdot), e^{i2\pi j\cdot}\rangle_{L^2[0,1]} \hat h_0(\xi+2\pi j) d\xi.
\end{split}
\notag
\ee
omitting the cutoff function $\a(\xi)$. This is to estimate $S(t)$ for $|x|>>Ct$ by using the stationary solutions $\bar u'h_+$ and $\bar u'h_-$ in Proposition \ref{prop S bigger than Ct}. Actually, our decomposition might work also in the analysis of \cite{JNRZ1}.
\end{remark}

\begin{proposition} \label{prop tilde S1}
Suppose $|\partial_x h_0(x)| \leq E_0(1+|x|)^{-r}$ for $r>1$ and sufficiently small $E_0>0$. Then if $|x|<<Ct$ for sufficiently large $C>0$,
\be
\begin{split}
|\tilde S_1(t)(\bar u' h_0)|
& \lesssim  \int_{-\infty}^{\infty} \Big[ (1+|x-y-at|+\sqrt t)^{-2}+t^{-\frac{1}{2}}e^{-\frac{|x-y-at|^2}{Mt}} \Big] |\partial_y h_0(y)| dy \\
& \lesssim E_0 \Big[ (1+|x-at|+\sqrt t)^{-r}+(1+t)^{-\frac{1}{2}}e^{-\frac{|x-at|^2}{M't}}   \Big ],
\notag
\end{split}
\ee
for sufficiently large $M'>M>0$.
\end{proposition}
\begin{proof}
Recalling $\check h_0(\xi,x) = \ds\sum_{j \in \ZZ}e^{i2\pi jx}\hat h_0(\xi+2\pi j)$, $\tilde S_1(t)$ is separated into four parts
\be
\begin{split}
\tilde S_1(t)(\bar u'h_0)
& = \int_{-\pi}^{\pi} e^{i\xi x}e^{\l(\xi)t}\a(\xi)\langle \tilde\phi(\xi,\cdot), \bar{u}^\prime \check h_0(\xi,\cdot) \rangle_{L^2[0,1]} d\xi
- \int_{-\infty}^{\infty} e^{i\xi x}e^{(-ia\xi-b\xi^2)t} \hat{h_0}(\xi) d\xi \\
& = \ds\sum_{j \in \ZZ\setminus\{0\}}\int_{-\pi}^{\pi}e^{i\xi x}e^{\l(\xi)t}\a(\xi)\langle \tilde\phi(\xi,\cdot)\bar{u}^\prime(\cdot), e^{i2\pi j\cdot}\rangle_{L^2[0,1]} \hat h_0(\xi+2\pi j) d\xi \\
& \qquad + \int_{-\pi}^{\pi}e^{i\xi x}e^{\l(\xi)t}\a(\xi)\langle \tilde\phi(\xi,\cdot)-\tilde \phi(0,\cdot), \bar u'\rangle_{L^2[0,1]} \hat h_0(\xi) d\xi \\
& \qquad + \int_{-\pi}^{\pi}e^{i\xi x}e^{(-ia\xi-b\xi^2)t}O(|\xi|^3t)\a(\xi)\hat h_0(\xi)d\xi \\
& \qquad + \int_{-\infty}^{\infty} e^{i\xi x}e^{(-ia\xi-b\xi^2)t}(\a(\xi)-1)\hat h_0(\xi)d\xi  \\
&  = I+ II+ III+ IV.
\notag
\end{split}
\ee

\textbf{Estimate I.} We first notice that by \eqref{eigenvalue from 0},
\be
e^{\l(\xi)t}=e^{-ia\xi-b\xi^2}e^{O(|\xi|^3)t}=e^{-ia\xi-b\xi^2}(1+O(|\xi|^3t)).
\notag
\ee
We separate again $I$ into two parts.
\be
\begin{split}
I
& =\ds\sum_{j \in \ZZ\setminus\{0\}}\int_{-\pi}^{\pi}e^{i\xi x}e^{\l(\xi)t}\a(\xi)\langle \tilde\phi(\xi,\cdot)\bar{u}^\prime(\cdot), e^{i2\pi j\cdot}\rangle_{L^2[0,1]} \hat h_0(\xi+2\pi j) d\xi \\
& = \ds\sum_{j \in \ZZ\setminus\{0\}}\int_{-\pi}^{\pi}e^{i\xi x}e^{\l(\xi)t}\frac{1}{\xi+2\pi j}\a(\xi)\widehat{\tilde{\phi}(\xi)\bar u'_j}^*\widehat{e^{i2\pi jx}\partial_xh_0(x)}(\xi) d\xi \\
& = \ds\sum_{j \in \ZZ\setminus\{0\}}\int_{-\pi}^{\pi}e^{i\xi x}e^{(-ia\xi -b\xi^2)t}\frac{1}{\xi+2\pi j}\a(\xi)\widehat{\tilde{\phi}(\xi)\bar u'_j}^*\widehat{e^{i2\pi jx}\partial_xh_0(x)}(\xi) d\xi \\
& \quad + \ds\sum_{j \in \ZZ\setminus\{0\}}\int_{-\pi}^{\pi}e^{i\xi x}e^{(-ia\xi -b\xi^2)t}O(|\xi|^3 t)\frac{1}{\xi+2\pi j}\a(\xi)\widehat{\tilde{\phi}(\xi)\bar u'_j}^*\widehat{e^{i2\pi jx}\partial_xh_0(x)}(\xi) d\xi \\
& = A + B.
\end{split}
\notag
\ee
 Here, $\widehat{\tilde{\phi}(\xi)\bar u'_j}^*$ denotes the complex conjugate of $j$-th Fourier coefficient of the $1$-periodic function $\tilde{\phi}(\xi)\bar u'_j$. In order to estimate $A$, we first set
\be
\hat{\b}_j:= \frac{1}{\xi+2\pi j}\a(\xi)\widehat{\tilde{\phi}(\xi)\bar u'_j}^*
\notag
\ee
which is bounded and in Schwartz class. Then we estimate $A$ as
\be
\begin{split}
|A|
& \leq \sum_{j \in \ZZ\setminus\{0\}} \Big|IFT(e^{(-ia\xi-b\xi^2)t})* \b_j * e^{-i2\pi j\cdot}\partial_x h_0 \Big|\\
& \leq  \Big|IFT(e^{(-ia\xi-b\xi^2)t}) \Big| * \sum_{j \in \ZZ\setminus\{0\}}|\b_j | * \Big| e^{-i2\pi j\cdot}\partial_x h_0 \Big| \\
& \approx \int_{-\infty}^\infty \frac{1}{\sqrt{4\pi bt}}e^{-\frac{|x-y-at|^2}{4bt}} \ds\sum_{j \in \ZZ\setminus\{0\}}|\b_j(y)|dy * |\partial_x h_0(x)|; \\
\end{split}
\notag
\ee
so it is enough to estimate $\ds\sum_{j \in \ZZ\setminus\{0\}}|\b_j(x)|$. Noting that
\be
|\widehat{x^2\b_j(x)}(\xi)|
=|\partial_{\xi}^2\hat{\b_j}(\xi)|
\leq \sum_{k=0}^2 \frac{1}{1+|j|}|\widehat{\partial_{\xi}^k \tilde{\phi}(\xi)\bar u_j'}|
\notag
\ee
and recalling $\hat{\b_j}(\xi)$ has a smooth cut-off function $\a(\xi)$, we have
\be
|x^2\b_j(x)|
=\Big| \int_{-\infty}^\infty e^{i\xi x}\widehat{x^2\b_j(x)}(\xi) d\xi \Big|
\leq \int_{|\xi|\leq \epsilon} | \partial_{\xi}^2\hat{\b_j}(\xi)| d\xi
\leq \sup_{|\xi|\leq \epsilon} \left(\sum_{k=0}^2 \frac{1}{1+|j|}|\widehat{\partial_{\xi}^k \tilde{\phi}(\xi)\bar u_j'}|\right).
\notag
\ee
Since $\partial_{\xi}\tilde{\phi}(\xi,\cdot)\bar u'(\cdot)$ and $\partial_{\xi}^2\tilde{\phi}(\xi,\cdot)\bar u'(\cdot)$ are also periodic, $\widehat{\partial_{\xi}\tilde{\phi}(\xi)\bar u'_j}$ and $\widehat{\partial_{\xi}^2\tilde{\phi}(\xi)\bar u'_j}$ are Fourier coefficients of $\partial_{\xi}\tilde{\phi}(\xi)\bar u'$ and $\partial_{\xi}^2\tilde{\phi}(\xi)\bar u'$, respectively. For each $k=0,1,2$, by using the Cauchy-Schwarz estimate,
\be
\begin{split}
\sum_{j\in \ZZ} \frac{1}{1+|j|} |\widehat{\partial_{\xi}^k\tilde{\phi}(\xi)\bar u'_j}|
& \leq \sqrt{ \sum_{j\in \ZZ} (1+|j|)^{-2} \sum_{j\in \ZZ}|\widehat{\partial_{\xi}^k\tilde{\phi}(\xi)\bar u'_j}|^2} \\
& \leq C\parallel\partial_{\xi}^k\tilde{\phi}(\xi)\bar u'\parallel_{L^2([0,1])};
\end{split}
\notag
\ee
and so
\be
\begin{split}
\ds\sum_{j \in \ZZ\setminus\{0\}}|\b_j(x)|
& \leq (1+|x|)^{-2}\sup_{|\xi|\leq \epsilon}\sum_{j \in \ZZ}\left(\sum_{k=0}^2 \frac{1}{1+|j|}|\widehat{\partial_{\xi}^k \tilde{\phi}(\xi)\bar u_j'}|\right) \\
& \leq (1+|x|)^{-2}\sup_{|\xi|\leq \epsilon} \sum_{k=0}^{2}\parallel \partial_{\xi}^k\tilde{\phi}(\xi)\bar u'\parallel_{L^2[0,1])} \\
& \leq C(1+|x|)^{-2}.
\notag
\end{split}
\ee
Thus, by Lemma \ref{algebraically decaying data},
\be
\begin{split}
|A|
& \lesssim  \Big[ (1+|x-at|+\sqrt t)^{-2}+(1+t)^{-\frac{1}{2}}e^{-\frac{|x-at|^2}{16bt}} \Big]* |\partial_xh_0(x)| \\
& = \int_{-\infty}^{\infty} \Big[ (1+|x-y-at|+\sqrt t)^{-2}+(1+t)^{-\frac{1}{2}}e^{-\frac{|x-y-at|^2}{16bt}} \Big] |\partial_y h_0(y)| dy.  \\
\end{split}
\notag
\ee
If $|\partial_y h_0(y)| \leq (1+|y|)^{-r}$ with $r>1$,  then
\be
|A| \lesssim E_0 \Big[ (1+|x-at|+\sqrt t)^{-r}+(1+t)^{-\frac{1}{2}}e^{-\frac{|x-at|^2}{Mt}}   \Big ]
\notag
\ee
for sufficiently large $M>0$. Here, we compute $\ds \int_{-\infty}^{\infty} (1+|x-y-at|+\sqrt t)^{-2}(1+|y|)^{-r}dy$ in two cases $|x-at|\leq \sqrt t$ and $|x-at|  > \sqrt t$ similarly as in Lemma \ref{algebraically decaying data}.

\bigskip

For $B$, we set
\be
\hat{\tilde \b}_j:= \frac{1}{\xi+2\pi j}\a^{\frac{1}{2}}(\xi)\widehat{\tilde{\phi}(\xi)\bar u'_j}^*;
\notag
\ee
so
\be
\begin{split}
|B|
& \leq \sum_{j \in \ZZ\setminus\{0\}} \Big|IFT\Big(e^{(-ia\xi-b\xi^2)t}O(|\xi^3|t)\a^{\frac{1}{2}}(\xi)\Big)* \tilde \b_j * e^{-i2\pi j\cdot}\partial_x h_0 \Big|\\
& \leq \Big| IFT\Big(e^{(-ia\xi-b\xi^2)t}O(|\xi^3|t)\a^{\frac{1}{2}}(\xi)\Big)\Big| * \sum_{j \in \ZZ\setminus\{0\}}|\tilde \b_j(x)| * |\partial_x h_0(x)| \\
& \leq \Big| \int_{|\xi| \leq \varepsilon} e^{i\xi x}\Big( \cdots \Big) d\xi +  \int_{\varepsilon \leq |\xi| \leq 2\varepsilon} e^{i\xi x}\Big( \cdots \Big) d\xi \Big| * \sum_{j \in \ZZ\setminus\{0\}}|\tilde \b_j(x)| * |\partial_x h_0(x)| \\
\end{split}
\notag
\ee
Computing similarly as in $A$, $\ds \sum_{j \in \ZZ\setminus\{0\}}|\tilde \b_j(x)| \lesssim (1+|x|)^{-2}$. For the second integration,
\be\label{greater than verep}
\begin{split}
\Big|\int_{\varepsilon \leq |\xi| \leq 2\varepsilon} e^{i\xi x}\Big( \cdots \Big) d\xi \Big|
& \leq  \int_{\varepsilon   \leq |\xi| \leq 2\varepsilon} e^{-b\xi^2 t}O(|\xi|^3t) d\xi \\
& \lesssim  t^{-\frac{1}{2}} e^{-b\frac{\varepsilon^2}{2} t} \leq t^{-\frac{1}{2}}e^{-\tau t} \leq t^{-\frac{1}{2}}e^{-\frac{|x-at|^2}{Mt}}
\end{split}
\ee
for some $\eta>0$ and sufficiently large $M>0$. Here, the last inequality is from the boundedness of  $\ds \frac{|x-at|}{t}$ because $\ds\frac{|x|}{t}$ is bounded, says $\ds \frac{|x-at|}{t} < S_1$ for some $S_1>0$. Indeed, for sufficiently large $M>0$,
\be\label{eta t}
e^{-\frac{|x-at|^2}{Mt}} =e^{-(\frac{|x-at|}{t})^2 \frac{t}{M}}  \geq e^{-\frac{S_1^2}{M}t} \geq e^{-\frac{\eta}{2}t}.
\ee
We now estimate $\ds \Big| \int_{|\xi| \leq \varepsilon} e^{i\xi x}\Big( \cdots \Big) d\xi \Big| $  by using complex contour integrals like \cite{J1}. Since $\ds \frac{|x|}{t}$ is bounded, we define
\be
\bar \a : = \frac{x-at}{2bt}
\notag
\ee
which is bounded and positive (without loss of generality). Thus, we have
\be
\begin{split}
\int_{|\xi| \leq \varepsilon} e^{i\xi x}\Big( \cdots \Big) d\xi
& = \int_{-\varepsilon}^{\varepsilon} e^{i(\xi_1+i\bar \a)(x-at) -b(\xi_1+i\bar \a)^2 t}O(|\xi_1+i\bar \a|^3 t)d\xi_1 \\
& \qquad \qquad + \int_0^{\bar \a} e^{i(\varepsilon+i\xi_2)(x-at) -b(\varepsilon+i\xi_2)^2 t}O(|\varepsilon+i\xi_2|^3 t)d\xi_2
\end{split}
\notag
\ee
which is bounded by $t^{-\frac{1}{2}}e^{-\frac{|x-at|^2}{Mt}}$ for sufficiently large $M>0$. Thus,
\be
\begin{split}
|B|
& \lesssim  \int_{-\infty}^{\infty} \Big[ (1+|x-y-at|+\sqrt t)^{-2}+(1+t)^{-\frac{1}{2}}e^{-\frac{|x-y-at|^2}{Mt}} \Big] |\partial_y h_0(y)| dy \\
& \lesssim E_0 \Big[ (1+|x-at|+\sqrt t)^{-r}+(1+t)^{-\frac{1}{2}}e^{-\frac{|x-at|^2}{Mt}}   \Big ].
\end{split}
\notag
\ee

\bigskip

\textbf{Estimate II, III and IV.} The estimate $II$, $III$ and $IV$ follows similarly. Since $\tilde\phi(\xi,\cdot)-\tilde \phi(0,\cdot)= O(|\xi|)$,
\be
\begin{split}
II
& = \int_{-\pi}^{\pi}e^{i\xi x}e^{\l(\xi)t}\a(\xi)O(1)(i\xi)\hat h_0(\xi) d\xi  \\
& = \int_{-\pi}^{\pi}e^{i\xi x}e^{\l(\xi)t}\a(\xi)O(1)\widehat{\partial_x h_0}(\xi) d\xi \\
& = \int_{-\infty}^\infty  \Big(\int_{-\pi}^{\pi}e^{i\xi (x-y)}e^{\l(\xi)t}\a(\xi)O(1) d\xi \Big) \partial_y h_0(y) dy \\
& = \int_{-\infty}^\infty  \Big(\int_{-\pi}^{\pi}e^{i\xi (x-y)}e^{(-ia\xi-b\xi^2)t}\a(\xi)O(1) d\xi \\
& \qquad \qquad  + \int_{-\pi}^{\pi}e^{i\xi (x-y)}e^{(-ia\xi-b\xi^2)t}O(|\xi|^3t)\a(\xi)O(1) d\xi  \Big) \partial_y h_0(y) dy\\
\end{split}
\notag
\ee
and
\be
\begin{split}
III
& = \int_{-\pi}^{\pi}e^{i\xi x}e^{(-ia\xi-b\xi^2)t}O(|\xi|^3t)\a(\xi)(i\xi)^{-1}i\xi \hat h_0(\xi)d\xi \\
& = \int_{-\infty}^\infty \Big(\int_{-\pi}^{\pi}e^{i\xi x}e^{(-ia\xi-b\xi^2)t}O(|\xi|^2t)\a(\xi)d\xi \Big)  \partial_y h_0(y)dy. \\
\end{split}
\notag
\ee
Computing similarly as in $I$, by complex contour integrals,
\be
\begin{split}
|II+III|
& \leq \int_{-\infty}^{\infty} t^{-\frac{1}{2}}e^{-\frac{|x-y-at|^2}{4bt}} |\partial_y h_0(y)| dy \\
& \lesssim E_0 \Big[ (1+|x-at|+\sqrt t)^{-r}+(1+t)^{-\frac{1}{2}}e^{-\frac{|x-at|^2}{Mt}}   \Big ].
\notag
\end{split}
\ee
Notice that
\be
\begin{split}
IV
& = \int_{-\infty}^{\infty} e^{i\xi x}e^{(-ia\xi-b\xi^2)t}(1-\a(\xi))(i\xi)^{-1}i\xi \hat h_0(\xi)d\xi \\
& = \int_{-\infty}^\infty \Big( \int_{|\xi| \geq \varepsilon }    e^{i\xi (x-y)}e^{(-ia\xi-b\xi^2)t}(1-\a(\xi))(i\xi)^{-1} d\xi \Big) \partial_y h_0(y) dy. \\
\end{split}
\notag
\ee
By \eqref{greater than verep}, we estimate $IV$; so $\ds |IV| \lesssim E_0 \Big[ (1+|x-at|+\sqrt t)^{-r}+(1+t)^{-\frac{1}{2}}e^{-\frac{|x-at|^2}{Mt}}   \Big ]$.

\end{proof}

We now treat $(\tilde S_2(t)(\bar u' h_0))(x)$ in terms of $\partial_x h_0(x)$ for $|x|<<Ct$ with sufficiently large $C>0$.
\begin{proposition}\label{prop tilde S2}
Suppose $|\partial_x h_0(x)| \leq E_0(1+|x|)^{-r}$ for sufficiently small $E_0>0$. Then if $|x|<<Ct$ for sufficiently large $C>0$,
\be
\begin{split}
|(\tilde S_2(t)(\bar u' h_0)(x)|
& \lesssim  \int_{-\infty}^{\infty} \Big[ (1+|x-y-at|+\sqrt t)^{-2}+t^{-\frac{1}{2}}e^{-\frac{|x-y-at|^2}{Mt}} \Big] |\partial_y h_0(y)| dy \\
& \lesssim E_0 \Big[ (1+|x-at|+\sqrt t)^{-r}+(1+t)^{-\frac{1}{2}}e^{-\frac{|x-at|^2}{M't}}   \Big ],
\end{split}
\ee
for sufficiently large $M'>M>0$.
\end{proposition}
\begin{proof}
Similarly as in $I$ of Proposition \ref{prop tilde S1}, re-express \eqref{tilde S2}:
\be
\begin{split}
(\tilde S_2(t)(\bar u' h_0))(x)
& = \sum_{j\in \ZZ}\int_{-\pi}^\pi e^{i\xi x}(e^{L_{\xi}t}(\bar u'e^{i2\pi jx}))(x)\frac{1-\a(\xi)}{i(\xi+2\pi j)}\widehat{e^{i2\pi jx}\partial_xh_0}(\xi) d\xi  \\
& \qquad + \sum_{j\in \ZZ}\int_{-\pi}^\pi e^{i\xi x}(e^{L_{\xi}t}\tilde \Pi(\xi)(\bar u'e^{i2\pi jx}))(x)\frac{\a(\xi)}{i(\xi+2\pi j)}\widehat{e^{i2\pi jx}\partial_xh_0}(\xi) d\xi \\
& \qquad + \sum_{j\in \ZZ} \int_{-\pi}^{\pi} e^{i\xi x}e^{\l(\xi)t} \frac{O(|\xi|)\widehat{\tilde{\phi}(\xi)\bar u'_j}^*\a(\xi)}{i(\xi+2\pi j)}  \widehat{e^{i2\pi jx}\partial_xh_0}(\xi) d\xi  \\
& = I +II +III.
\end{split}
\notag
\ee
Setting $\hat{\bar\b}_j:= \ds\frac{\xi}{\xi+2\pi j}\a^{\frac{1}{2}}(\xi)\widehat{\tilde{\phi}(\xi)\bar u'_j}^*$, we estimate $III$ as before, where we are using the complex contour integrals:
\be
\begin{split}
|III|
& \lesssim   \Big| \int_{-\infty}^{\infty} e^{i\xi x} e^{(-ia\xi-b\xi^2)t}\left(O(1)+O(|\xi|^3t)\right)\a^{\frac{1}{2}}(\xi)d\xi \Big| * \sum_{j \in \ZZ} |\bar{\b_j}(x)|dx * | \partial_x h_0(x) |
\end{split}
\notag
\ee
which is bounded by $\ds\int_{-\infty}^\infty t^{-\frac{1}{2}}e^{-\frac{|x-y-at|^2}{Mt}}|\partial_y h_0(y)|dy$ as before. The only difference compared to term $I$ in Proposition \ref{prop tilde S1} is that the summation contains $j=0$, but it is totally okay because we have $\ds\frac{\xi}{\xi+2\pi j}$ in $\hat{\bar\b}_j$ instead of $\ds\frac{1}{\xi+2\pi j}$.

Now we consider $I$ and $II$ which are same estimations. Re-expressing $I$,
\be
\begin{split}
I
& = \int_{-\pi}^\pi e^{i\xi x}(1-\a(\xi))(e^{L_{\xi}t}(\bar u'\check h_0))(\xi,x)d\xi \\
& = \sum_{j \in \ZZ}\int_{-\pi}^\pi e^{i\xi x}(e^{L_{\xi}t}(\bar u'e^{i2\pi jx}))(x)\frac{1-\a(\xi)}{i(\xi+2\pi j)}\widehat{e^{-i2\pi jx}\partial_xh_0}(\xi)d\xi \\
& = \sum_{j \in \ZZ}\int_{-\infty}^\infty e^{i\xi x}(e^{L_{\xi}t}(\bar u'e^{i2\pi jx}))(x)\frac{(1-\a(\xi))\chi_{[-\pi,\pi]}(\xi)}{i(\xi+2\pi j)}\widehat{e^{-i2\pi jx}\partial_xh_0}(\xi)d\xi.
\end{split}
\notag
\ee
Set
\be
d_j(\xi,x,t):=(e^{L_{\xi}t}(\bar u'e^{i2\pi jx}))(x)
\notag
\ee
which is periodic in $x$ on $[0,1]$ and setting $c_{j,k}(\xi,t)$ are Fourier coefficients of $d_j$, we have
\be
\begin{split}
I
& = \sum_{j,k \in \ZZ}e^{i2\pi kx} \int_{-\infty}^\infty e^{i\xi x}c_{j,k}(\xi) \frac{(1-\a(\xi))\chi_{[-\pi,\pi]}(\xi)}{i(\xi+2\pi j)}\widehat{e^{-i2\pi jx}\partial_xh_0}(\xi)d\xi \\
& =  \sum_{j,k \in \ZZ}e^{i2\pi kx} \int_{-\infty}^\infty e^{i\xi x}c_{j,k}(\xi) \frac{(1-\a(\xi))\chi_{[-\pi,\pi]}(\xi)}{i(\xi+2\pi j)}d\xi * e^{-i2\pi jx}\partial_xh_0(x) \\
& = \sum_{j \in \ZZ} \int_{-\pi}^\pi e^{i\xi x}d_j(\xi,x,t) \frac{1-\a(\xi)}{i(\xi+2\pi j)}d\xi * e^{-i2\pi jx}\partial_xh_0(x), \\
\end{split}
\notag
\ee
and so
\be
|I| \lesssim \sup_{\varepsilon<|\xi|< \pi}\sum_{j \in \ZZ} \Big|d_j(\xi,x,t) \frac{1-\a(\xi)}{i(\xi+2\pi j)}  \Big| * |\partial_x h(x)|;
\notag
\ee
thus, it is enough to estimate $\ds \sum_{j \in \ZZ} \Big|d_j(\xi,x,t) \frac{1-\a(\xi)}{i(\xi+2\pi j)}  \Big|$ independently on $\xi$. By Cauchy-Schwarz inequality,
\be
\sum_{j \in \ZZ} \Big|d_j(\xi,x,t) \frac{1-\a(\xi)}{i(\xi+2\pi j)}  \Big| \leq \sqrt{\sum_{j \in \ZZ}\frac{1}{(1+|j|)^2}\sum_{j \in \ZZ}|d_j(\xi,x,t)|^2 } \leq C\sqrt{\sum_{j \in \ZZ}|d_j(\xi,x,t)|^2}.
\notag
\ee
Noting that $Re\sigma(\L_{\xi})\leq -\eta <0$ for any $|\xi| \geq \varepsilon$, we re-define the sector as $\Omega \cap \{Re \l \leq -\eta \}$ independently on $\xi$ and set
\be
\Gamma=\partial(\Omega \cap \{Re \l \leq -\eta \}).
\notag
\ee
Then we have
\be
\begin{split}
\sum_{j \in \ZZ}|d_j(\xi,x,t)|^2
& = \sum_{j \in \ZZ}\Big|(e^{L_{\xi}t}(\bar u'e^{i2\pi j \cdot}))(x)\Big|^2 \\
& = \sum_{j \in \ZZ} \Big | \int_{\Gamma}e^{\l t}(L_{\xi}-\l)^{-1}(\bar u'(x)e^{i2\pi j(x)}) d\l \Big|^2 \\
& = \sum_{j \in \ZZ} \Big | \int_{\Gamma}e^{\l t}\left(\int_0^1 [G_{
\xi,\l}(x,z)]\bar u'(z)e^{i2\pi jz}dz\right) d\l \Big|^2 \\
& \leq  \left( \sum_{j \in \ZZ} \Big | \int_{\Gamma}e^{\l t}\left(\int_0^1 [G_{
\xi,\l}(x,z)]\bar u'(z)e^{i2\pi jz}dz\right) d\l \Big| \right)^2,  \\
\end{split}
\notag
\ee
where the brackets $[\cdot]$ denote the periodic extensions of the given function onto the whole line. Since  $[G_{\xi,\l}(x,z)]\bar u'(z)$ is periodic in $z$ on $[0,1]$, let's set
\be
h_{j}(\xi,x,\l):=\int_0^1 e^{-i2\pi jz}[G_{\xi,\l}(x,z)]\bar u'(z)dz
\notag
\ee
which are Fourier coefficients of $[G_{\xi,\l}(x,z)]\bar u'(z)$. Recall that  $|G_{\xi,\l}(x,z)|\leq C|\l|^{-\frac{1}{2}}$ and $|\partial_zG_{\xi,\l}(x,z)| \leq C$ in \cite{J1}, for $|\l|>R$, $R$ sufficiently large $R$, and  $|G_{\xi,\l}(x,z)|, |\partial_zG_{\xi,\l}(x,z| \leq C$, for $|\l|<R$. Then we have
\be
\begin{split}
\sum_{j \in \ZZ}|h_j^*(\xi,x,\l)|
& \leq C\sqrt{\sum_{j \in \ZZ }(1+|j|)^{-2}\sum_{j \in \ZZ }(1+|j|^2)|h_j(\xi,x,\l)|^2} \\
& \leq C\|[G_{\xi,\l}(x,z)]\bar u'(z)\|_{H^1{\{ z;[0,1]} \}} \\
& \leq C,
\end{split}
\notag
\ee
where * denote complex conjugate.
Thus we have
\be
\begin{split}
\sqrt{\sum_{j \in \ZZ}|d_j(\xi,x,t)|^2}
& =  \int_{\Gamma} e^{Re \l t} \sum_{j \in \ZZ}|h_j(\xi,x,\l)|d\l \\
& \leq Ce^{-\eta t}\int_0^\infty e^{-\theta kt}dk \\
& \leq Ct^{-1}e^{-\eta t} \\
& \leq Ct^{-1}e^{-\frac{\eta t}{2}}e^{-\frac{|x-at|^2}{Mt}},
\end{split}
\notag
\ee
for some $\eta>0$ and large $M>0$. Here, the last inequality is from \eqref{eta t} again.
\end{proof}

\section{Nonlinear iteration scheme} \label{section NIS}

Recalling the nonlinear perturbation equation \eqref{nonlinear perturbation equation} and \eqref{v}, we now define $\psi(x,t)$ to cancel $E(x,t;y)$ and $\bar u's_*^p(t)$ in $S(t)v_0$ and $S(t)\bar u'h_0$, respectively,
\be \label{psi}
\psi(x,t):= s_*^p(t)(\bar u' h_0)+\int_{-\infty}^\infty E(x,t;y)v_0(y)dy + \int_0^t \int_{-\infty}^{\infty} E(x,t-s;y)\mathcal{N}(y,s)dyds.
\ee
Since there is a cutoff function in $E$, $\psi(x,0)=h_0(x)$ and so we have a new integral representation of $v(x,t)$:
\be \label{v2}
v(x,t)= \tilde S_*(t)(\bar u' h_0)+\int_{-\infty}^{\infty}\tilde G(x,t;y)v_0(y)dy+\int_0^t\int_{-\infty}^{\infty} \tilde G(x,t-s;y)\mathcal{N}(y,s)dyds.
\ee

\begin{remark}
Similarly as in the localized case, we define $\psi$ as ``bad" terms which have not enough decay rates in the solution operator $S(t)$ to close a nonlinear iteration. One can actually see $\psi_x \sim v$. Since $\mathcal{N}$ consists of $v$ and derivatives of $\psi$, by \eqref{psi} and \eqref{v2}, we prove Theorem \ref{main theorem} in the next section.
\end{remark}

\section{Nonlinear stability} \label{section NS}

We now prove the main theorem, starting with the following lemma.

\begin{lemma}\label{lemma zeta}
For $r \geq \frac{3}{2}$ and sufficiently small $E_0>0$, we assume
\be
\begin{split}
 |\tilde u_0(x -h_0(x))-\bar u(x)| + & \sum_{k=1, 2} |\partial_x^k h_0(x)|+|h_0(x)-h_{\pm \infty}|  \leq E_0(1+|x|)^{-r}, \\
|h_{+\infty}|=|h_{-\infty}|\leq E_0  & \quad \text{and} \quad v_0:=\tilde u_0(x -h_0(x))- \bar u(x)  \in H^2(\RR).
\end{split}
\notag
\ee
For $v$ and $\psi$ defined in Section \ref{section NIS}, we define
\be
\zeta(t):=\sup_{0\leq s \leq t, ~ x \in \RR} |(v,\psi_t, \psi_x, \psi_{xx})(x,s)|\Big[ (1+|x-as|+\sqrt s)^{-r} + (1+s)^{-\frac{1}{2}}e^{-\frac{|x-as|^2}{Ms}}\Big]^{-1},
\notag
\ee
for sufficiently large $M>0$. Then for all $t\geq 0$ for which $\zeta(t)$ is finite, we have
\be
\zeta(t) \leq C(E_0+\zeta(t)^2)
\notag
\ee
for some constant $C>0$.
\end{lemma}

\smallskip

\begin{proof} For any $0\leq s \leq t$, applying Propositions \ref{prop S bigger than Ct}, \ref{prop tilde S1} and \ref{prop tilde S2} to the integral representation of $v$ in \eqref{v2},
\be
|\tilde S_*(s)(\bar u'h_0)|\leq CE_0\Big[(1+|x-as|+\sqrt s)^{-r} + (1+s)^{-\frac{1}{2}}e^{-\frac{|x-as|^2}{Ms}}\Big]
\notag
\ee
for any $x\in \RR$; so then for any $0\leq s \leq t$ and any $x\in \RR$,
\be
|\tilde S_*(s)(\bar u'h_0)|\Big[(1+|x-as|+\sqrt s)^{-r} + (1+s)^{-\frac{1}{2}}e^{-\frac{|x-as|^2}{Ms}}\Big]^{-1} \leq CE_0.
\notag
\ee
Reminding the pointwise bounds of $\tilde G(x,t;y)$ in Theorem \ref{PGB}, we have for any $0\leq s \leq t$ and any $x\in \RR$,
\be
\begin{split}
\Big| \int_{-\infty}^\infty \tilde G(x,s;y)v_0(y)dy \Big|
& \leq \int_{-\infty}^\infty |\tilde G(x,s;y)||v_0(y)|dy \\
& \leq  \int_{-\infty}^\infty \Big((1+s)^{-1}+s^{-\frac{1}{2}}e^{-\eta s}\Big) e^{-\frac{|x-y-as|^2}{Ms}} (1+|y|)^{-r} dy \\
& \leq CE_0\Big[(1+|x-as|+\sqrt s)^{-r} + (1+s)^{-\frac{1}{2}}e^{-\frac{|x-as|^2}{Ms}}\Big].
\end{split}
\notag
\ee
Recalling $\eqref{Q} \sim \eqref{T}$ and applying the boundedness of $|v_x|_{L^{\infty}}$ in the main theorem of \cite{JNRZ1} to $\mathcal{N}$ in \eqref{v2}, we have
\be
\begin{split}
|(Q, R, S, T)(x,s)|
& \leq C|(v, \psi_t, \psi_{x}, \psi_{xx})(x,s)|^2 \\
& \leq C\zeta^2(t)\Big[ (1+|x-as|+\sqrt s)^{-r} + (1+s)^{-\frac{1}{2}}e^{-\frac{|x-as|^2}{Ms}} \Big]^{2}.
\end{split}
\notag
\ee
By using integration by parts in the third term of $v$,
\be
\begin{split}
& \Big| \int_0^t \int_{-\infty}^\infty \tilde G(x, t-s;y)\mathcal{N}(y,s)dyds \Big| \\
& \leq \int_0^t \int_{-\infty}^\infty |\tilde G_y(x, t-s;y)||(Q, R, S, T)(y,s)|dyds \\
& \leq C\zeta^2(t)\int_0^t \int_{-\infty}^\infty (t-s)^{-1}e^{-\frac{|x-y-a(t-s)|^2}{M(t-s)}}(1+|y-as|+\sqrt s)^{-2r}dyds\\
& \quad \quad + C\zeta^2(t)\int_0^t \int_{-\infty}^\infty (t-s)^{-1}e^{-\frac{|x-y-a(t-s)|^2}{M(t-s)}}(1+s)^{-1}e^{-\frac{|y-as|^2}{Ms}}dyds\\
& = C\zeta^2(t)(I + II).
\end{split}
\notag
\ee
Noting first that
\be
\int_{-\infty}^\infty e^{-\frac{|x-y-a(t-s)|^2}{M(t-s)}}e^{-\frac{|y-as|^2}{Ms}}dy \leq Ct^{-\frac{1}{2}}s^{\frac{1}{2}}(t-s)^{\frac{1}{2}} e^{-\frac{|x-at|^2}{Mt}},
\notag
\ee
we easily estimate $II$ as
\be
II \leq e^{-\frac{|x-at|^2}{Mt}}  \int_{0}^{t} (t-s)^{-\frac{1}{2}}(1+s)^{-\frac{1}{2}} t^{-\frac{1}{2}} ds \lesssim (1+t)^{-\frac{1}{2}} e^{-\frac{|x-at|^2}{Mt}}.
\notag
\ee
We now estimate $I$ by separating $\ds\int_0^t$ into $\ds\int_{0}^{t/2}$ and $\ds\int_{t/2}^{t}$. By Lemma \ref{algebraically decaying data}, we have
\be
\begin{split}
&\int_{-\infty}^\infty (t-s)^{-\frac{1}{2}}e^{-\frac{|x-(y-as)-at|^2}{M(t-s)}}(1+|y-as|)^{-r} dy \\
& \quad \quad \lesssim (1+|x-at|+\sqrt{t-s})^{-r}+(1+t-s)^{-\frac{1}{2}} e^{-\frac{|x-at|^2}{M(t-s)}};
\end{split}
\notag
\ee
so then
\be
\begin{split}
\int_{0}^{t/2}
& \lesssim \int_{0}^{t/2} (1+s)^{-\frac{r}{2}}(t-s)^{-\frac{1}{2}}\Big[(1+|x-at|+\sqrt{t-s})^{-r}+(1+t-s)^{-\frac{1}{2}} e^{-\frac{|x-at|^2}{M(t-s)}}\Big] ds \\
& \lesssim \Big[(1+|x-at|+\sqrt t)^{-r}+(1+t)^{-\frac{1}{2}}e^{-\frac{|x-at|^2}{Mt}}\Big]\int_{0}^{t/2}(1+s)^{-\frac{r}{2}}(t-s)^{-\frac{1}{2}} ds \\
& \lesssim (1+|x-at|+\sqrt t)^{-r}+(1+t)^{-\frac{1}{2}}e^{-\frac{|x-at|^2}{Mt}}.
\end{split}
\notag
\ee
In order to estimate $\ds\int_{t/2}^{t}$, we re-prove Lemma \ref{algebraically decaying data} as
\be
\begin{split}
& \int_{-\infty}^{\infty} (t-s)^{-\frac{1}{2}}e^{-\frac{|x-(y-as)-at|^2}{M(t-s)}}(1+|y-as|+\sqrt s)^{-2r} dy  \\
& \quad \leq (1+|x-at|+\sqrt s)^{-2r} + (1+t-s)^{-\frac{1}{2}}e^{-\frac{|x-at|^2}{M(t-s)}}\int_0^\infty (1+y+\sqrt s)^{-2r} dy \\
& \quad \leq (1+|x-at|+\sqrt s)^{-2r} + (1+t-s)^{-\frac{1}{2}}e^{-\frac{|x-at|^2}{Mt}}(1+\sqrt s)^{-2r+1}.
\end{split}
\notag
\ee
Since $r\geq \frac{3}{2}$, we obtain
\be
\begin{split}
& \int_{t/2}^{t} (t-s)^{-\frac{1}{2}} \Big[(1+|x-at|+\sqrt s)^{-2r} + (1+t-s)^{-\frac{1}{2}}(1+\sqrt s)^{-2r+1} e^{-\frac{|x-at|^2}{Mt}}\Big] ds \\
& \quad \leq  (1+|x-at|+\sqrt t)^{-r} + (1+t)^{-r+\frac{1}{2}+\frac{1}{2}}e^{-\frac{|x-at|^2}{Mt}} \\
& \quad \leq  (1+|x-at|+\sqrt t)^{-r} + (1+t)^{-\frac{1}{2}}e^{-\frac{|x-at|^2}{Mt}}.
\end{split}
\notag
\ee
Similarly, by $\eqref{first x derivative of psi} \sim\eqref{first t derivative of psi}$ and Theorem \ref{PGB}, we estimate $\psi_t$, $\psi_x$ and $\psi_{xx}$ as
\be
|(\psi_t, \psi_x, \psi_{xx})(x,s)| \leq C(E_0+\zeta(t)^2)\Big[ (1+|x-as|+\sqrt s)^{-r} + (1+s)^{-\frac{1}{2}}e^{-\frac{|x-as|^2}{Ms}}\Big],
\notag
\ee
for any $0\leq s \leq t$ and $x \in \RR$; thus we complete the proof.
\end{proof}

From here, the proof of Theorem \ref{main theorem} goes similarly as in the localized case.

\begin{proof}[\textbf{Proof of Theorem \ref{main theorem}}]
Without loss of generality, we assume $C>\frac{1}{2}$. Then $\zeta (0) < 2CE_0$. Since $\zeta(t)$ is continuous so long as it remains small, by the continuous induction, $\zeta(t) < 2CE_0$ for all $t\geq 0$ if $E_0<\frac{1}{4C^2}$. Indeed, if $\zeta(t)=2CE_0$, then by Lemma \ref{lemma zeta}, $2CE_0 \leq C(E_0+4C^2E_0^2)$; so $E_0\geq \frac{1}{4C^2}$ which is a contradiction.

\end{proof}

\end{document}